\documentclass[12pt,epsfig,tikz,multi]{amsart}\usepackage{latexsym}
\usepackage{tikz}
\usepackage[nocenter]{qtree}
\usepackage{breqn}
\usepackage{amsrefs}
\usepackage{float}
\restylefloat{table}
\usepackage[edges]{forest}
\usepackage{amsmath,amscd,amssymb}
\usepackage{graphicx}
\usepackage{caption}
\usepackage{subcaption}
\usepackage{enumitem}
\usepackage{xcolor}
\usepackage{graphicx,wrapfig,lipsum}
\usepackage{adjustbox}
 2

\usepackage[utf8]{inputenc}
\usepackage[T1]{fontenc}

\include{plaquettemacro}
\usepackage{hyperref}
\newtheorem{theorem}{Theorem}[section]
\newtheorem{prop}{Proposition}[section]

\newtheorem{cor}[theorem]{Corollary}
\newtheorem{definition}[theorem]{Definition}
\newtheorem{example}[theorem]{Example}

\numberwithin{equation}{section}
\newtheorem*{theorem*}{Theorem}
\title{Swallowtails and cone-like singularities on a maxface }
\author{Pradip Kumar}
\address{Department of Mathematics, Shiv Nadar University, Dadri 201314, Uttarpradesh, India.}
\email{pradip.kumar@snu.edu.in}
\author{Anu Dhochak}
\address{Department of Mathematics, Shiv Nadar University, Dadri 201314, Uttarpradesh, India.}
\email{ad404@snu.edu.in}
\date{}

\begin{document}
\maketitle
\begin{abstract} When a connected component of the set of singular points of the maxface $X$  consists of  only generalized cone-like singular point, we construct a sequence of maxfaces $X_n$, with an increasing number of swallowtails, converging to the maxface $X$.  We include the general discussion toward this. 
\end{abstract}
\section{Introduction}
Maximal surfaces in the Lorentz-Minkowski space $\mathbb E_1^3$ are space-like immersions that maximize area locally.
They are similar to the minimal surfaces in $\mathbb R^3$, as both are zero mean curvature surface and can be constructed using many similar methods but in  the maximal surfaces non-isolated singularities appear. 

Estudillo and Romero \cite{Estudillo1992} named generalized maximal immersion for the maximal surfaces with singularities.  They are of two types, branched and non branched. Non-branched maximal immersions are the one when limiting tangent plane contains a light-like vector.  The Maxface, introduced by Umehara and Yamada \cite{UMEHARA2006}, are non branched generalized maximal immersions.

Umehara and Yamada \cite{UMEHARA2006} have shown that all the maxface, as a map to $\mathbb R^3$, become frontal, and a few of them near the singularity become front. As a front and frontal, in \cite{fujimori2015}, \cite{Fujimori2009}, \cite{Fujimori2007}, \cite{UMEHARA2006}, \cite{teramoto2020gaussian}, we see authors have discussed various singularities: swallowtails, cuspidal-edge, cuspidal cross cap, cuspidal butterflies, cuspidal $S_1^{-}$. There is another important singularity, called the  cone-like singularity. Cone-like singularities,  introduced by  O. Kobayashi  \cite{KOBAYASHI1984},  are non isolated  and discussed in \cite{Fernandez2005a}, \cite{Fujimori2009} etc.

In \cite{Kim2006}, Kim and Yang gave an example of family of maxfaces for each natural number $n$,  that has swallowtails in an increasing order. We call it (a name given by the authors in \cite{Fujimori2009}) Kim and Yang's toroidal maxface.  The authors discussed the global properties of this toroidal maxface.  The family of toroidal maxfaces seems to ``converge" to the cone-like. Moreover, in \cite{Fujimori2009}, the authors talked about trinoids with swallowtails whose computer graphics look cone-like.   All these poses a question, if we start from a maxface having a cone-like singularity (or generalized cone-like) then, can we find a sequence of maxfaces having an increasing number of swallowtails, and converging to the former one?  

In this article, we discuss the general situation, we start from a maxface with generalized cone-like singularity, and we construct the sequence in  various cases.   Let $\lambda$ be a non-constant real analytic curve, whose trace is compact, when $\lambda$ satisfies certain conditions, in section 4, we prove the following theorem. 
\begin{theorem*}[Theorem: \ref{thm:main}]
Let $X$ be a maxface having generalized cone-like singularity on $trace$ of $\lambda$, with singular Bj\"{o}rling data $\{\lambda,\alpha, \beta\}$. Moreover the data  $\{\lambda,\alpha, \beta\}$ has a sequence of the scaling functions as in the definition \ref{lemma:f_n}. Then there is a sequence of maxfaces $X_n$ defined on a neighborhood $\Omega$ of trace of $\lambda$, having an increasing number of swallowtails, and sequence of maxfaces converges (in the norm $\|\;.\;\|_\Omega$) to $X$, having cone-like.
\end{theorem*}

A significant class of maxface (defined locally) can be constructed using the singular Bj\"{o}rling problem introduced by Kim and Yang in \cite{Kim2007} and also discussed in \cite{RPR2016}.   Given a null curve $\gamma$ and a null vector field on the real axis or the unit circle,  in  \cite{RPR2016}, \cite{Kim2007}, authors have constructed the maxface.  In section 3, we start with the singular Bj\"{o}rling problem when we expect singularities lie on the trace of a non-constant smooth curve $\lambda$.  Further in section 3, we give necessary and sufficient conditions (in propositions: \ref{Prop:Cone-likeLambda}, \ref{prop:Swallotail}) on the singular Bj\"{o}rling data (defined on the trace of $\lambda$) such that singularity is of swallowtail or generalized cone-like.

In section 4, we prove theorem \ref{thm:main} and an example of a sequence of maxfaces having an increasing number of swallowtails converging to the Lorentzian Catenoid.    All the discussion is local that is near a singular curve.

%The discussion of this article is close to \cite{Fujimori2009}, \cite{Kim2006}, \cite{Kim2007}.

\begin{figure}[t]
\centering
\includegraphics[scale=0.5]{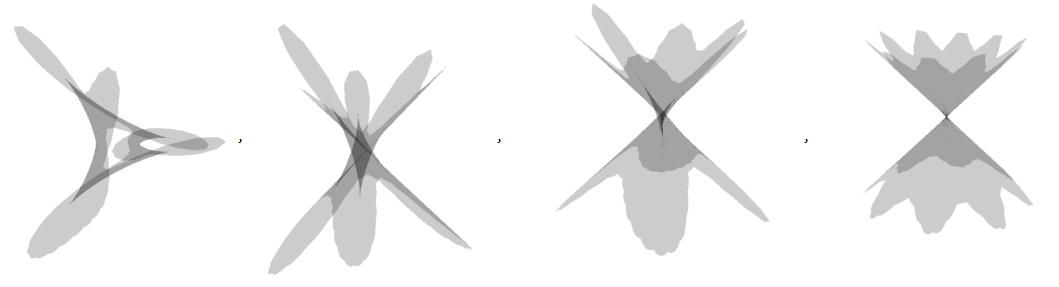}
\caption{Maxfaces converging to the Lorentzian Catenoid}
\label{fig:introduction}
\end{figure}

\section{Preliminaries}
The Lorentz-Minkowski space $\mathbb E_1^3$ is a vector space $\mathbb R^3$ with metric $\langle\; ,\;\rangle:\mathbb R^3 \times \mathbb R^3\to \mathbb R$ defined by $\langle (a_1, b_1, c_1), (a_2, b_2, c_2)\rangle:=a_1a_2+b_1b_2-c_1c_2,
$
where $(a_1, b_1, c_1)$ and $(a_2, b_2, c_2)$ are two vectors in $\mathbb R^3.$  In the following, we write the Weierstrass-Enneper representation of the maxface given by Umehara and Yamada \cite{UMEHARA2006}.

\subsection{Weierstrass Enneper representation} Let $M$ be a Riemann surface, the map $X:M\to \mathbb E_1^3$ be a maxface if and only if there is a pair $(g, f)$
 of meromorphic function and a holomorphic 1-form on $M$ such that $(1+|g|^2)^2|f|^2$ gives a positive definite Riemannian metric on $M$, and $|g|$ is not identically equal $1$.  Moreover for the map  $
\Phi:=\left((\frac{1}{2}(1+g^2), \frac{i}{2}(1-g)^2, -g)f\right)
$,  $
Re\int_{\gamma_j}\Phi=0, \;\;\;\;(j=1,..,N)$
for loops $\{\gamma_j\}_{j=1}^N$ such that $[\gamma_j]$ are generators of $\pi_1(M)$  and
$X(p)=Re\int_o^p\Phi$.

\subsection{Singularities}
Various singularities appear on the maxface. Here we review a few of these. We start by recalling the definition of  $\mathcal{A}-$ equivalence  as in \cite{Fujimori2007}, \cite{UMEHARA2006}.

Two smooth maps $X: \Omega\subset\mathbb{R}^2 \to \mathbb E_1^3$ and $Y:\mathcal{V}\subset \mathbb{R}^2\to \mathbb E_1^3$ are said to be $\mathcal{A}-$equivalent at the points $p\in\Omega$ and $q\in \mathcal{V}$ if there exists a local diffeomorphism $\eta$ of $\mathbb{R}^2$ with $\eta(p)=q$ and a local diffeomorphism $\Phi$ of $\mathbb E_1^3$ with $\Phi(X(p))=Y(q)$ such that $Y=\Phi\circ X \circ \eta^{-1}$.

\begin{definition}[Swallowtails \cite{Fujimori2007}, \cite{UMEHARA2006}]
A maxface $X:\Omega\subset \mathbb C\to\mathbb E_1^3$  is said to have a swallowtail at $p\in \Omega$ if at $p\in \Omega$ and $(0,0)\in \mathbb{R}^2$, $X$ is $\mathcal{A}$-equivalent to
$
f_{sw}(u,v)=(3u^4+u^2v,\, 4u^3+2uv,\, v) \,;\, (u,v)\in \mathbb{R}^2$.
\end{definition}

\begin{definition}[Shrinking singularity \cite{Kim2007}]
For a maxface $X: \Omega\to \mathbb E_1^3$,
we call $p\in \Omega$ a shrinking singular point if there is some $U_p \subset \Omega$ and a regular embedded curve $\gamma: I \subset U_p\to \Omega$, such that every point of $\gamma[I]$ singularity and $X \circ \gamma$ is a single point.
\end{definition}

\begin{definition}[Generalized cone-like singularity and cone-like singularity \cite{Fujimori2009}]
Let $\sum_0$ be the connected component of the set of all admissible singular points on a generalized maximal immersion $X:\Omega\to \mathbb E_1^3$. Each point of $\sum_0$ is called a generalized cone-like singular point if $\sum_0$ is compact and the image $X(\sum_0)$ is a single point. Moreover, if there is a neighborhood $U$ of $\sum_0$ such that $X(U\setminus \sum_0)$ is embedded, then each point of $\sum_0$ is called a cone-like singular point.
 \end{definition}

If the trace of $\lambda$ is compact and it is a connected component of the singularity, then the shrinking singularity becomes generalized cone-like. For examples and detailed discussion, we refer to \cite{Fujimori2007}, \cite{Kim2007}, \cite{UMEHARA2006} etc.

For a maxface  $X: \Omega\to \mathbb E_1^3$ with the Weierstrass data $\{g,f\}$, the authors in \cite{Fujimori2009}, \cite{Kim2007},  \cite{UMEHARA2006} have given a very useful criterion to check the nature of singularity at a particular point.  For a maxface, if $p$ is a singularity, then there exists a curve $\lambda$ on a neighborhood of $p$ such that every point on the trace of $\lambda$ is a singularity.  The curve $\lambda$ is said to be a singular curve and the vector $\eta$ along $\lambda$ is said to null curve if $dX_{\lambda(t)}(\eta)=0$.
Umehara and Yamada \cite{UMEHARA2006} have shown that in a maxface, such direction is uniquely determined and calculated in terms of Weierstrass data.

Following \cite{UMEHARA2006}, in terms of the Weierstrass data $\{g,f\}$, and in terms of  $\eta$, $\lambda$,  we  recall the  functions:
\begin{definition}[\cite{UMEHARA2006}]\label{DefnAH}
 On a neighborhood $U_p$ of a singularity $p$, we define  $A(z):= \frac{g'}{g^2 f}$.

Since $\lambda$ is a singular curve parameterized on $I$, we define $H(t)= Det(\lambda^\prime(t), \eta(t))$ for all $t \in I$.
\end{definition}

Moreover\cite{UMEHARA2006} a maxface becomes front on a neighborhood of singularity if $Re(A(\lambda(t))\neq 0$ at the singular points  and as a front, singularity at $p=\lambda(t)$ is swallowtail if and only if $H^\prime(t)\neq 0$.

\section{Singular Bj\"{o}rling problem and the necessary and sufficient conditions for few singularities}
In this section, we will discuss the singular Bj\"{o}rling problem in a general case when we expect the singularities to lie on the trace of a smooth non-constant curve. It is similar to the discussion in \cite{RPR2016}, \cite{Kim2007}. Further, we will find the necessary and sufficient conditions on the singular Bj\"{o}rling data so that singularities are of generalized cone-like or swallowtails.
\subsection{Singular Bj\"{o}rling problem on a curve}
We start by defining the singular Bj\"{o}ring data similar to the way Kim and Yang have given in  \cite{Kim2007}.

\begin{definition}[Singular Bj\"{o}rling data]\label{defn:singularBjorlingdata} This is a triplet $\{ \lambda,\alpha,\beta\}$, where
\begin{enumerate}
    \item $\lambda: I\to \mathbb{C}$ be a smooth curve such that for all $t\in I$, $\lambda^\prime(t)\neq 0$.
    \item $\alpha$ be a null curve and $\beta $ be a null vector field defined on $I$ and for all $t\in I$, $\langle \alpha^\prime(t), \beta(t)\rangle=0$. Moreover, either $\alpha^\prime(t)\neq 0$ or $\beta(t)\neq 0$ for all $t \in I$.
    \item  The map given by $\phi(\lambda(t))= \alpha'(t) - i\beta(t)$ has an extension $\phi(z)$ on the trace of $\lambda$.
    \item Analytic extension $g(w)$ of  $g(\lambda(t))= G(t)$  where
\begin{equation}\label{defn:Gt}
G(t):=\begin{cases}
\frac{\alpha_1^\prime(t)+i\alpha_2^\prime(t)}{\alpha_3^\prime(t)}\; \text{ if } \alpha^\prime(t)\neq 0\\
\frac{\beta_1(t)+i\beta_2(t)}{\beta_3(t)} \text{ if } \beta_3(t)\neq 0
\end{cases}
\end{equation}
and $|g(w)|$ is not identically equal to 1.
\end{enumerate}

\end{definition}

For the data $\{\lambda,\alpha,\beta\}$, if   $Re\int_{\gamma_j} \phi(z) dz$ vanishes for all loops $\gamma_j$ on a neighborhood  $\Omega$ of trace of $\lambda$. Then the map $X:= (X_1,X_2,X_3) : \Omega\to \mathbb E_1^3$, given by
\begin{equation}\label{eqn:NEWsolSingularBjorline}
X_{\lambda,\alpha,\beta}(z)= Re\left( \int_{\lambda(t_0)}^{z} \phi(w) dw\right)
\end{equation}
is well defined and it gives a generalized maximal immersions on $\Omega$, containing the trace of $\lambda$.  Moreover the trace of $\lambda$ is a subset of the singular set.    The generalized maximal immersion given in the equation \ref{eqn:NEWsolSingularBjorline} is  unique up to translations (we have many choices of $\lambda(t_0)$).

\begin{example}
Let $\lambda(t)= (t,0)$, and $\alpha$ and $\beta$ be real analytic curve and vector field respectively.  Then $\phi(z)= \alpha^\prime(z)-\beta(z)$ is analytic and it gives a generalized maximal immersion. This is the case of the singular Bj\"{o}rling problem as introduced and discussed by Kim and Yang  in \cite{Kim2007}.
\end{example}
\begin{example}
Let $\lambda(t)= e^{it}$, $0\leq t \leq 2 \pi$. Let $\alpha^\prime(t)= (\cos(t), \sin(t),1)$ and $\beta(t)= (0,0,0)$, then
\begin{align*}
 \phi(e^{it})&= (\cos(t), \sin(t), 1)\\
&=\left(\dfrac{e^{it}+e^{-it}}{2}, \dfrac{e^{it}-e^{-it}}{2i},1\right)
\end{align*}
it has analytic extension $\phi(z)= (\frac{z}{2}+\frac{1}{2z}, i(\frac{1}{2z}-\frac{z}{2}),1)$
in a annular region containing unit circle.   It gives a generalized maximal immersion in an annular region that contains the unit circle. 
\end{example}

The generalized maximal immersion as in the equation \ref{eqn:NEWsolSingularBjorline} is a maxface on a neighborhood of trace of $\lambda$. Moreover, for $z=\lambda(t)$,  $f(z) = \phi_1(z) -i\phi_2(z)
 =(\alpha_1^\prime(t)+\beta_2(t)) -i(\alpha_2^\prime(t)+\beta_1(t))$. Therefore $f$ is never zero on Trace of $\lambda$.

The Gauss map is given by
\[g(z)= -\frac{\phi_3(z)}{\phi_1(z)-i\phi_2(z)}=-\frac{\phi_1(z)+i\phi_2(z)}{\phi_3(z)},
\]
and
$g(\lambda(t))= G(t)$, as in the equation \ref{defn:Gt}.
\subsection{Null direction}
Let $X_{\lambda,\alpha,\beta}: \Omega\to \mathbb E_1^3$ be the maxface for a fixed $t_0$ as in the equation \ref{eqn:NEWsolSingularBjorline}.  Since
$dX= \frac{1}{2}X_z dz+ \frac{1}{2} X_{\overline{z}}d\overline{z}= \phi(z)dz+\overline{\phi(z)} d\overline{z}$.
At $z= \lambda(t)$, we have $$dX_{\lambda(t)}= (\alpha^\prime(t)-i\beta(t))dz_{\lambda(t)}+(\alpha^\prime(t)+i\beta(t)) d\bar{z}_{\lambda(t)}$$

So let $v= \eta\frac{\partial}{\partial \bar{z}}+\bar{\eta}\frac{\partial}{\partial\bar{z}}$ such that $dX_{\lambda(t)}(v)=0$, we have
 $(\alpha_3^\prime(t)-i\beta_3(t))\eta+(\alpha_3^\prime(t)+i\beta_3(t)) \overline {\eta}=0
$, that is:
$$\alpha_3^\prime(t).Re(\eta)+\beta_3(t) Im(\eta)=0.$$
Therefore the null direction at each $\lambda(t)$ is given by

\begin{equation}\label{prop:nullDirection}
    \eta(t)= {-\beta_3(t)+i\alpha_3^\prime(t)}.
\end{equation}

\subsection{Calculation for $A(t)$ and $H^\prime(t)$}
In the following, we will calculate $A(\lambda(t))$ and $H^\prime(t)$ as given in the equation \ref{DefnAH}. We write it here again   $$A(\lambda(t))= \frac{G^\prime(t)}{G^2(t)\lambda^\prime(t)f(\lambda(t)},\text{ and } H(t)= Det(\lambda^\prime(t), \eta(t)). $$

We define
\begin{equation}\label{DB_12}
D(\beta_{12}, \beta_{12}^\prime):= \beta_1\beta_2^\prime-\beta_2\beta_1^\prime,\;\; D(\alpha_{12}^\prime, \alpha_{12}^{\prime\prime}):= \alpha_1^\prime\alpha_2^{\prime\prime}-\alpha_2^\prime\alpha_1^{\prime\prime}.
\end{equation}
We have
$g^\prime(\lambda(t)=\frac{G^\prime(t)}{\lambda^\prime(t)}$, where $G$ is given by the equation \ref{defn:Gt}.

For the case when  $\beta_3(t)\neq 0$,
a straight calculation gives,

\[\frac{G^\prime(t)}{G^2(t).f(\lambda(t))}= \frac{-D(\beta_{12},\beta_{12}^\prime)}{(\alpha_3^{\prime 2}+\beta_3^2)\beta_3}+i\frac{\alpha_3^\prime D(\beta_{12},\beta_{12}^\prime)}{(\alpha_3^{\prime 2}+\beta_3^2)\beta_3^2}.\]

Therefore we
\begin{align*}
    A(\lambda(t))&= \left(\frac{g^\prime(\lambda(t)}{g^2(\lambda(t) f(\lambda(t))}\right)=\frac{G^\prime(t)}{G^2(t).f(\lambda(t))}.\frac{\lambda_1^\prime(t)-i\lambda_2^\prime(t)}{\lambda_1^{\prime 2}(t)+\lambda_2^{\prime 2}(t)}\\
    &= \frac{D(\beta_{12},\beta_{12}^\prime)}{(\alpha_3^{\prime 2}+\beta_3^2)\beta_3^2.(\lambda_1^{\prime 2}(t)+\lambda_2^{\prime 2}(t))}.(-\beta_3+i\alpha_3^\prime).(\lambda_1^\prime(t)-i\lambda_2^\prime(t))\\
    &=\frac{D(\beta_{12},\beta_{12}^\prime)}{(\alpha_3^{\prime 2}+\beta_3^2)\beta_3^2.(\lambda_1^{\prime 2}+\lambda_2^{\prime 2})}\left((-\beta_3\lambda_1^\prime+\alpha_3^\prime\lambda_2^\prime)+i. (\alpha_3^\prime\lambda_1^\prime+\beta_3\lambda_2^\prime) \right)
\end{align*}
Similarly,  we can calculate for the case  when   $\alpha_3^\prime(t)\neq 0$ and  we get the following.
\begin{equation}\label{defn:A(t)}
A(\lambda(t)):=\begin{cases}
 \frac{D(\alpha_{12}^\prime,\alpha_{12}^{\prime\prime})}{(\lambda_1^{\prime 2}+\lambda_2^{\prime 2})(\alpha_3^{\prime 2}+\beta_3^2)\alpha_3^{\prime 2}}\left((-\beta_3\lambda_1^\prime+\alpha_3^\prime\lambda_2^\prime)+i.(\alpha_3^\prime\lambda_1^\prime+\beta_3\lambda_2^\prime)\right)\; \text{ if } \alpha^\prime_3(t)\neq 0\\
\frac{D(\beta_{12},\beta_{12}^\prime)}{(\alpha_3^{\prime 2}+\beta_3^2)\beta_3^2.(\lambda_1^{\prime 2}+\lambda_2^{\prime 2})}\left((-\beta_3\lambda_1^\prime+\alpha_3^\prime\lambda_2^\prime)+i. (\alpha_3^\prime\lambda_1^\prime+\beta_3\lambda_2^\prime) \right)\text{ if } \beta_3(t)\neq 0.
\end{cases}
\end{equation}

Now we calculate the function  $H(t)$.  The null direction is given by the equation \ref{prop:nullDirection}. We get
$H(t)= \lambda_1^\prime\alpha_3^\prime+\lambda_2^\prime\beta_3$
and therefore we have
\begin{equation}\label{h(t)}
H^\prime(t)= \lambda_1^{\prime\prime}\alpha_3^\prime+\lambda_1^\prime\alpha_3^{\prime\prime}+\lambda_2^\prime\beta_3^\prime+\lambda_2^{\prime\prime}\beta_3
\end{equation}
\subsection{For shrinking and  generalized cone-like}
The maxface $X_{\lambda,\alpha, \beta}$ as in the equation \ref{eqn:NEWsolSingularBjorline} has shrinking singularity on the trace of $\lambda$ if and only if  $X_{\lambda,\alpha, \beta}(\lambda(t))$ is constant and $\;\;g^\prime(\lambda(t))\neq 0 \;\ $ for all $t\in I$, where $g$ is the Gauss map.

This implies $\forall t \in I$, $DX_{\lambda(t)}(\lambda'(t)) = 0$ if and only if $\forall t \in I, \;\alpha^\prime(t)\lambda_1^\prime(t)+ \beta(t)\lambda_2^\prime(t)=0$ if and only if for all $t\in I, Im(A(t)=0$.

As $g^\prime(\lambda(t)) =\dfrac{G'(t)}{\lambda'(t)} \neq 0$ and $Im(A(t))=0$,
   if and only if $\forall t \in I, Re (A(t)) \neq 0$.

From above we get, $\forall t \in I$, $-\beta_3\lambda'_1 + \alpha_3'\lambda_2' \ne 0$. Therefore $Re(A(t))\neq 0$     if and only if  $D(\beta_1,\beta_2^\prime) \neq 0$ when $\beta_3\neq 0$, and $D(\alpha_{12}^\prime, \alpha_{12}^{\prime\prime})\neq 0$ when $\alpha_3^\prime(t)\neq 0$.

Summarising everything here, we get the following.
\begin{prop}\label{Prop:Cone-likeLambda}
The solution of Bj\"{o}rling problem (as in the equation \ref{eqn:NEWsolSingularBjorline}) on a neighborhood of the trace of $\lambda$, consists of only shrinking singularities on the trace of $\lambda$  if and only if
\begin{enumerate}
 \item $\forall t\in I, \alpha^\prime_3(t)\lambda_1^\prime(t)+ \beta_3(t)\lambda_2^\prime(t)=0$
\item $D(\beta_1,\beta_2^\prime)\neq 0$ when $\beta_3(t)\neq 0$, and  $D(\alpha_{12}^\prime, \alpha_{12}^{\prime\prime})\neq 0$ when $\alpha_3^\prime(t)\neq 0$.  Here $D(.\;,\;.)$ as in the equation \ref{DB_12}.
\end{enumerate}
In particular, if the trace of $\lambda$ is compact, then the trace of $\lambda$ consists of only generalized cone-like singularities.
\end{prop}

\begin{example}\label{exampls:cone-likeonunitcircle}When $\lambda(t)=e^{it}$, $0\leq t<2\pi$,
Let $\alpha^\prime(t)=\cos t(\cos t, \sin t, -1)$, and $\beta(t)= \sin t (\cos t, \sin t, -1)$. For this singular Bj\"{o}rling data, the maxface as in the equation \ref{eqn:NEWsolSingularBjorline}, satisfies conditions of the proposition \ref{Prop:Cone-likeLambda}. Moreover the trace of $\lambda$ is unit circle. Therefore every point of the unit circle is a generalized cone-like singularity (in fact it is cone-like singularity).

The Weierstrass data $\{g,w\}$ for this maxface is $\{g(z)= z, w(z)= \frac{1}{z^2}\}$. This is the Lorentzian Catenoid \cite{KOBAYASHI1983}.
\end{example}
\begin{example}\label{shrinkingbaseexample}
Let $\lambda(t)= (t,0), a<t<b$, and  $\alpha(t)$ be a constant curve.  Let $\beta$ be a real null analytic vector field defined on $I$ such that $\forall t$,$ (\beta_1^\prime\beta_2-\beta_2\beta_1^\prime) \neq 0$ then solution as in equation \ref{eqn:NEWsolSingularBjorline} gives shrinking singularity.

In particular, if $\beta (t)=(1-t^2, 2t, 1+t^2) $ then we see that for all $t\in\mathbb R$, $(\beta_1\beta_2^\prime-\beta_2\beta_1^\prime)\neq 0$. Therefore, this $\beta$ with any constant $\alpha$ gives a maxface.

Moreover, the Weierstrass data for this maxface (equation \ref{eqn:NEWsolSingularBjorline}) is
$$\left\{g(z)=-\frac{1-z^2+2iz}{1+z^2}, w(z)= -2z-i(1-z^2)\right\}$$

\end{example}

To treat separately the case, when $\lambda(t)= (t,0)$ and every point $(t,0)$ is a shrinking singularity, we write the following particular case as the  corollary.
\begin{cor}\label{cor:shrinking}
Let $\lambda(t)= (t,0)$ then the solution of singular Bj\"{o}rling problem as in the equation \ref{eqn:NEWsolSingularBjorline} consist of shrinking singularity on a neighborhood of trace of $\lambda$ if and only if
\begin{enumerate}
 \item $\alpha$ is a constant curve.
\item and for all t,$ (\beta_1^\prime\beta_2-\beta_2\beta_1^\prime) \neq 0$
\end{enumerate}
\end{cor}

\subsection{For Swallowtails}
Here we give necessary and sufficient conditions on the singular Bj\"{o}rling data as in the definition \ref{defn:singularBjorlingdata}, such that $X_{\lambda,\alpha,\beta}$ has swallowtails at some points on the trace of $\lambda$

Let $\{\lambda,\alpha, \beta\}$ be the singular Bj\"{o}rling data and $X_{\lambda,\alpha,\beta}$ be its solution as in the equation \ref{eqn:NEWsolSingularBjorline} on a neighborhood of trace of $\lambda$.  We denote it by $X$.

We know at $p=\lambda(t)$, $X$ has swallowtails if and only if at $p=\lambda(t)$,
Re$(A(\lambda(t))\neq 0$,  Im$(A(\lambda(t)))= 0$, and $H^\prime(t)\neq 0$.  Here $A$ and $H$ are functions as in the definition \ref{DefnAH}. At point $p=\lambda(t)$,  $A(\lambda(t))$ and $H^\prime(t)$ are given by equations \ref{defn:A(t)} and \ref{h(t)}

Putting everything together, we have the following criterion.

\begin{prop}\label{prop:Swallotail}
At $\lambda(t)$, solution of the Singular Bj\"{o}rling problem as in the equation \ref{eqn:NEWsolSingularBjorline} has Swallowtails if and only if at $t$
\begin{enumerate}
\item $\alpha_3^\prime\lambda_1^\prime+\beta_3\lambda_2^\prime= 0$ and
\item $\lambda_1^{\prime\prime}\alpha_3^\prime+\lambda_1^\prime\alpha_3^{\prime\prime}+\lambda_2^\prime\beta_3^\prime+\lambda_2^{\prime\prime}\beta_3\neq 0$
\item when $\alpha^\prime(t)\neq 0$, $D(\alpha_{12}^\prime,\alpha_{12}^{\prime\prime})\neq 0$, and when $\beta^\prime(t)\neq 0$, $D(\beta_{12},\beta_{12}^{\prime})\neq 0$
\end{enumerate}
\end{prop}

In particular,for $\lambda(t)= (t,0)$ and trace of $\lambda$ has shrinking singularity, we have
\begin{cor}\label{swallotailin(t,0)}
Solution of the singular Bj\"{o}rling problem , when $\lambda(t)=(t,0)$,  has swallowtails at $\lambda(t)$ if and only if at $t$

$\alpha_3^\prime=0$, $\beta_3$ and $\alpha^{\prime \prime}_3$ are not zero and
 $ D(\beta_{12},\beta_{12}^{\prime})\neq 0$
\end{cor}

With the conditions in the propositions \ref{Prop:Cone-likeLambda}, \ref{prop:Swallotail},  in the following,  we have given a family of maxfaces $X_n$ defined in an  annular region containing the unit circle as the singularity. Moreover for each $n$ there is an increasing number of swallowtails on the unit circle.

\begin{example}\label{examp:Seqnswallotail1}
For $n\geq 1$, consider the Bj\"{o}rling data on  $\lambda(t)= e^{it},\; 0\leq t\leq 2\pi.$
as $\alpha_n^\prime(t)= \cos t(\cos t,\sin t,-1)$   and $\beta_n(t)= \left(\sin t+\frac{\cos nt}{n}\right)(\cos t,\sin t,-1)$.

We have at each $z= e^{it}$
\begin{align*}
    \alpha_{n3}^\prime\lambda_1^\prime+\beta_{n3}\lambda_2^\prime&= -\frac{1}{n}\cos t \cos nt\\
    \lambda_1^{\prime\prime}\alpha_{n3}^\prime+\lambda_1^\prime\alpha_{n3}^{\prime\prime}+\lambda_2^\prime\beta_{n3}^\prime+\lambda_2^{\prime\prime}\beta_{n3}&= \cos t\sin nt +\frac{1}{n}\sin t \cos nt.
\end{align*}
From the proposition \ref{prop:Swallotail}, we see 
\begin{figure}[htp]
\includegraphics[scale=0.3]{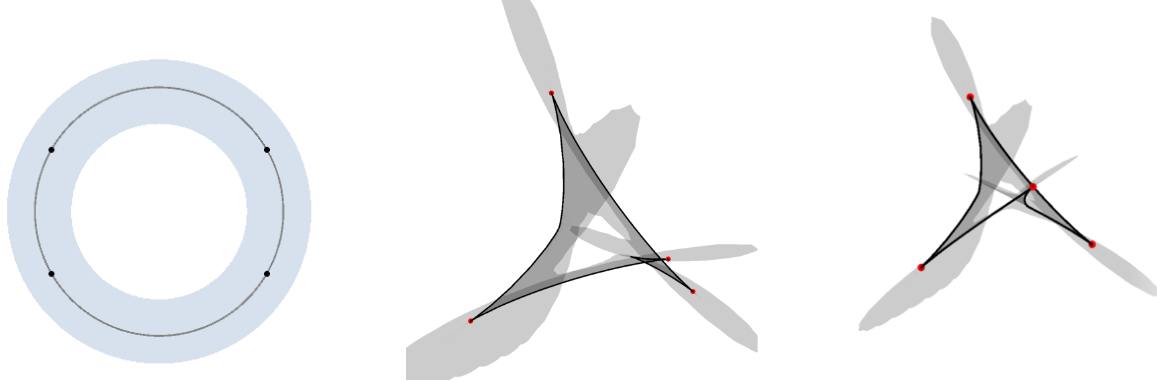}\\
\includegraphics[scale=0.3]{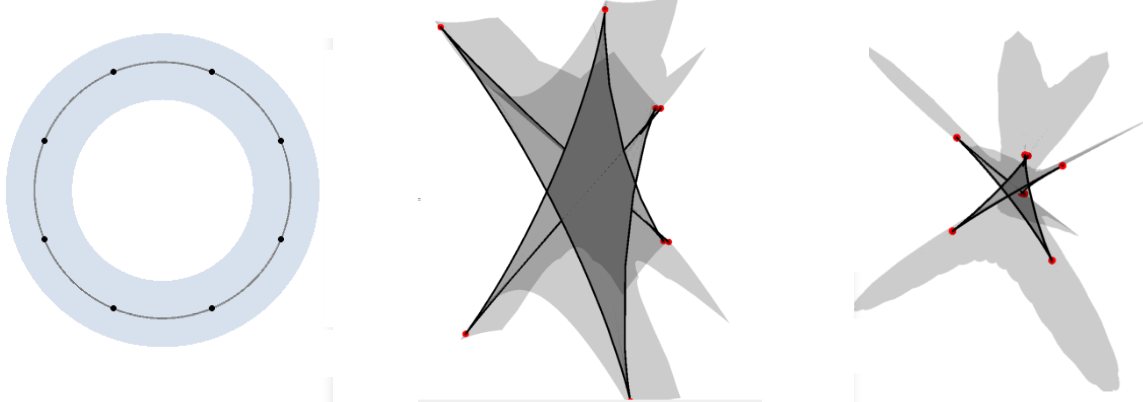}\\
\includegraphics[scale=0.23]{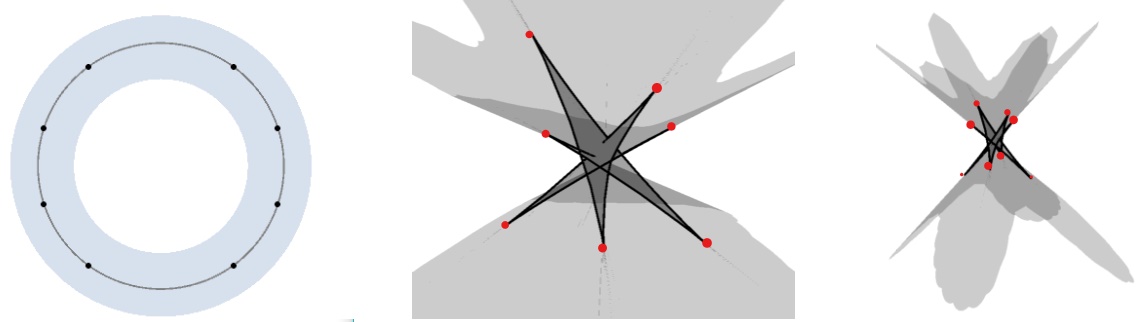}\\
\;\;\;\;\;\;\includegraphics[scale=0.31]{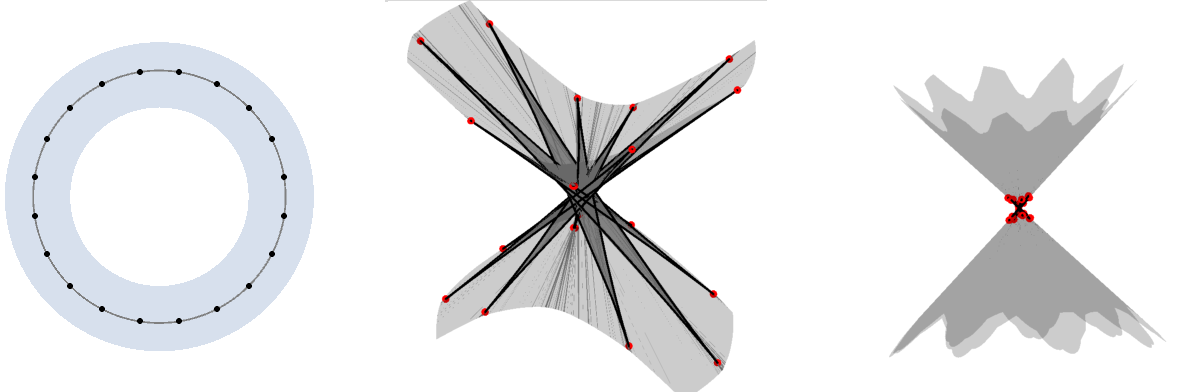}
\caption{A sequence of maxfaces with an increasing number of swallowtails as in the example \ref{examp:Seqnswallotail1}, for n=3, 4, 5 and 10.}
\label{fig:seqofSwallowtils}

\end{figure}
\begin{itemize}
   \item when $n$ is even then $\forall z = e^{it_k}$ where $t_k= \dfrac{(2k+1)\pi}{2n},\;\; k=0,1,....,2n-1$, are swallowtails for these maxfaces and
   \item when $n$ is odd then $\forall z = e^{it_k}$ where $t_k= \dfrac{(2k+1)\pi}{2n},\;\; k=\{0,1,....,2n-1\}\setminus\{\frac{n-1}{2},\frac{3n-1}{2}\}$ are  swallowtails,  So there are in total 2n-2 swallowtails.
\end{itemize}

The Weierstrass data for each $X_{\lambda,\alpha_n,\beta_n} $ is given by
$$\left\{g_n(z)= z, w_n(z)= \left(\frac{1}{z^2}-\frac{i}{2nz}(z^n+z^{-n})\right)\right\}$$

 In the figure \ref{fig:seqofSwallotails}, for various $n$, we have shown the maxfaces in an annular region where these are defined. The dotted points on the unit circle are the points where the respective maxface has swallowtails. In the middle, we have shown the curve where $|z|=1 $ is mapped in $\mathbb E_1^3$ with singular points. The last column is for the  maxfaces for $n=3,4,5,$ and $10$.  This figure captures the essence of this article.

\end{example}

From the above picture, we see this sequence of maxfaces ``converges"  to the  Lorentzian Catenoid.   This example is similar (at least graphically) to Kim and Yang's Toroidal maxface \cite{Kim2006} and Trinoids as in   \cite{Fujimori2009}.
All these motivate us to study the general problem, we talked in the introduction. 

\section{Sequence of swallowtails converging to the generalized cone-like}

Let $\Omega\subset \mathbb C$ be a simply connected domain, $E= \mathbb R^n$ or $\mathbb C^n$ and $X\in C(\overline{\Omega}, E)$, the space continuous maps. For each $z\in \overline{\Omega}$, we denote, $\|X(z)\|= \max\{|X_i(z)|: i=1\ldots n\}$ and  $\|X\|_{\Omega}:=\sup_{z\in \overline{\Omega}}\|X(z)\|.$  The space $C(\overline{\Omega}, E)$ becomes a Banach space under the norm $\|\;.\;\|_{\Omega}$.  For convergence of sequence of maxfaces, we will be using these norms.

\subsection{For the shrinking singularities on $\lambda(t)= (t,0)$}
In the following,  we see that for a maxface $X$ with the Bj\"{o}rling data $\{\lambda(t)= (t,0), \alpha, \beta\}$ and having shrinking singularities on each $(t,0)$,  we have a sequence of maxfaces $X_n$,  with an increasing number of swallowtails, converging to $X$ in a domain containing $\lambda(t)$. Precisely we prove the following-

\begin{prop}
Let $X:\mathcal{U}\to \mathbb E_1^3$ be a maxface, for a singular Burling data $\{\lambda, \alpha, \beta\}$ defined on $\lambda: [a,b]\to \mathcal{U},\; t\to (t,0)$, and $X$ has shrinking singularities on $(t,0)$. Then there is a domain $\Omega$ containing trace of $\lambda$ and a sequence of maxfaces $X_n$, such that
\begin{enumerate}
 \item $X$ and $X_n$ defined on $\Omega$
 \item Each $X_n$ has $n$ swallowtails on the trace of $\lambda$, moreover
\item  $X_n$ converges to $X$ in $\|\;.\;\|_{\Omega}$
\end{enumerate}
\end{prop}
\begin{proof}
Without loss of generality, we take $I=[0,1]$. The maxface $X_{\lambda,\alpha,\beta}$ for the data
$\{\lambda, \alpha,\beta\}$ has shrinking singularities each $(t,0)$, therefore $\alpha$ and $\beta$ satisfy the conditions given in Proposition \ref{Prop:Cone-likeLambda}.

We take $\Omega$ such that $I\subset \Omega\subset \overline{\Omega}\subset \mathcal{U}$ and   $\overline{\Omega}\subset \{z\in \mathbb C:
|z|<r\}$ for some $r>0$.  Since $I$ is closed interval, such $\Omega$ exits.

For natural number $n\geq 1,$  we take a polynomial,
\begin{equation}\label{eqn:functionf_n}
 f_n(z):= \frac{1}{(r+1)^n.n}\prod_{i=1}^n (z-\frac{1}{i})
\end{equation}

We see that for each $z\in \overline{\Omega}$,    $|f_n(z)|<\frac{1}{n}$. Therefore $f_n\to zero$ uniformly on $\overline{\Omega}$.  For each $t\in [0,1]$, we define
\[
\alpha_n(t)= \int_{t_0}^tf_n(t)\beta(t),\;\;\beta_n(t)= \beta(t).\]

For each $n$ ,$\{t\to (t,0), \alpha_n,\beta_n\}$ satisfies to be the data of singular Bj\"{o}rling problem. Moreover at each $t_i= \frac{1}{i}$, these satisfies all the conditions in the corollary \ref{swallotailin(t,0)}.

Therefore each maxface $X_n$ (the solution of singular Bj\"{o}rling problem for $\{\alpha_n, \beta_n\}$) has swallowtails at each $t_i=\frac{1}{i}$.

Moreover  for $z\in \overline{\Omega}$, $\alpha_n^\prime(z)= f_n(z)\beta(z)$ and  each $\beta_i(z)$ is bounded on $\overline{\Omega}$.  Hence on $\overline{\Omega}$, $\|\alpha_n^\prime\|_\Omega \to zero$.

Let fix some point $t_0\in I$ and for $z\in \overline{\Omega}$, let  $|\int_{t_0}^z dw|<L$. We denote $h^k$ for $k$-th co-ordinate function for $h= (h^1,h^2, h^3)$.   

We have , $\forall z\in \overline{\Omega}$,
\begin{align*}
    |X^k_n(z)-X^k(z)|&=\left|Re\int_{t_0}^z(\alpha_n^{k\prime}-i\beta_n^k)-(\alpha^{k\prime}-i\beta^k(t)) dw\right|\\
    &\leq \left|\int_{t_0}^z(\alpha^{k\prime}_n-\alpha^{k\prime})dw\right|+\left|\int_{t_0}^z(\beta_n^k-\beta^k)dw\right|\\
    &\leq L.\left( \|\alpha^\prime_n-\alpha^\prime\|_\Omega+\|(\beta_n-\beta)\|_\Omega\right)
\end{align*}

Therefore we have
$$\|X_n-X\|_\Omega\leq  L.\left( \|\alpha^\prime_n-\alpha^\prime\|_\Omega+\|(\beta_n-\beta)\|_\Omega\right)\leq L\|\alpha_n^\prime\|_\Omega \to zero. $$

This proves that $X_n$ converges to $X$.
\end{proof}
Now in the following, we will discuss a similar situation for arbitrary $\lambda$.  For that, we give the following definition.

\begin{definition}\label{lemma:f_n} We say a sequence $\{f_n\}$, of functions defined on $I$,  \textbf{ a sequence of  scaling functions for the singular Bj\"{o}rling data $\{\lambda,\alpha,\beta\}$ } if for each $n$, there are  points $t_{n_k}\in I$, $k= 0,1\ldots a_n$,  and  $a_n\geq n$ such that,
\begin{enumerate}
\item On the set $\{t_{n_k}: k=0,1,..,a_n\}$, $\lambda$ is 1-1.
 \item $\forall k= 0, 1,.. a_n,\;\; f_n(t_{n_k})=0$ and $ f_n^\prime(t_{n_k})\neq 0$.
 \item $g_n(\lambda(t))= f_n(t)(\alpha^\prime(t)-i\beta(t))$ has analytic extension on a neighborhood $\Omega$, of the trace of $\lambda$.
\item $Re \int_{\gamma_i}g_n =0$ for each closed curve $\gamma_i$ in the domain $\Omega$ as in (3).
\item $\|g_n\|_\Omega\to zero$.
\end{enumerate}
\end{definition}

For arbitrary data $\{\lambda,\alpha,\beta\}$, it is not direct to say whether such function exists or not, but for the particular cases, we can say.

\begin{example}
Let the   singular Bj\"{o}rling data be as in the example \ref{shrinkingbaseexample} with $\lambda(t)= (t,0)$ on $[0,1]$.  Then the sequence of functions as defined in the equation \ref{eqn:functionf_n} on $I$ with  $t_{n_k}=\frac{1}{k};\;\;k=  1,.., n$ turns out to be the sequence of scaling functions.
\end{example}

\begin{example}\label{f_nexamplet}
Let $\lambda(t)= e^{it},\;0\leq t\leq 2\pi$ and  $\alpha$, $\beta$ as in the example \ref{exampls:cone-likeonunitcircle}.  We take $f_n(t)= \frac{1}{2^{n+2}.n}\cos nt$.
Let $\Omega=  Ann(0, \frac{1}{2}, 2):=\{\frac{1}{2}< |z|<2\}.$

We see, for all $n$,  $f_n$ is zero at each $t_{n_k}=(2k+1)\frac{\pi}{2n};\; k=0,\ldots, 2n-1$. Moreover $f_n^\prime(t_{n_k})\neq 0$. Here $a_n= 2n-1>n$.  We have
\begin{align*}
    g_n(e^{it})&= \frac{1}{2^{n+2}n}\cos nt(\alpha^\prime(t)-i\beta(t))\\
    &= \frac{1}{2^{n+2}n} e^{-it}\left(\cos t, \sin t, -1\right)\cos(nt)\\
    &=\frac{1}{2^{n+2}n} e^{-it}\left(\frac{e^{it}+e^{-it}}{2}, \frac{e^{it}-e^{-it}}{2i}, -1\right)\left(\frac{e^{int}+e^{-int}}{2}\right).
\end{align*}
Each $g_n(e^{it})$ has analytic extension $g(z)$ in $z= e^{it}$. Moreover $Re \int_{\gamma_i}g_n(z)dz=0$ for each closed curve $\gamma_i$ in $\Omega$.

We denote, $g_n= (g_{n1}, g_{n2}, g_{n3})$,  for all $z\in \overline{\Omega}$, we see,
for each $i= 1, 2 ,3$,  $|g_{ni}(z)|<\frac{1}{n}$, therefore $\|g_n(z)\|<\frac{1}{n}$. It gives $\|g_n\|_\Omega=\sup_{z\in\overline{\Omega}}\|g_{n}(z)\|<\frac{1}{n}\to 0$
\end{example}

\subsection{For general singular Bj\"{o}rling problem}
Let $\lambda$ be a curve whose trace is compact and $\alpha$ and $\beta$ are the singular Bj\"{o}rling data as in the definition \ref{defn:singularBjorlingdata}, that have cone-like singularity on the trace of $\lambda$, so data $\{\lambda, \alpha, \beta\}$ satisfies condition of the Proposition \ref{Prop:Cone-likeLambda}.

Since we have for all $t\in I$  $\alpha_3^\prime\lambda_1^\prime+\beta_3\lambda_2^\prime=0$, $(\alpha_3^\prime(t),\beta_3(t))\neq (0,0)$, $(\lambda_1^\prime, \lambda_2^\prime)\neq (0,0)$, therefore we have for all $t$,
\begin{equation}\label{eqnmain2}
\forall t\in I,\;\;
-\beta_3(t)\lambda_1^\prime(t)+\alpha_3^\prime(t)\lambda_2^\prime(t)\neq 0.
\end{equation}

Suppose there is a sequence of the scaling functions for the $\{\lambda,\alpha,\beta\}$, as in the definition \ref{lemma:f_n}.  We  define
\[\alpha_n^\prime(t)= \alpha^\prime(t)+f_n(t)\beta(t),\;\;
\beta_n(t)= \beta(t)-f_n(t)\alpha^\prime(t)\]

For any $t$, both $\alpha_n^\prime$ or $\beta_n$ are not zero.  Moreover,
\[ \psi_n(\lambda(t))= \alpha_n^\prime(t)-i\beta_n(t)
=\alpha^\prime(t)-i\beta(t)+f_n(t)(\beta(t))+i\alpha^\prime(t))\]
Since for the Bj\"{o}rling data, $\{\lambda, \alpha, \beta\}$, we have $f_n$ and $g$ as in the definition \ref{lemma:f_n}. We get
\[\psi_n(\lambda(t))=\phi(\lambda(t))-ig_n(\lambda(t)).\]

These $\psi_n$ have analytic extension in  variable $z=\lambda(t)$. For each $n$ the triplet $\{\lambda, \alpha_n, \beta_n\}$ satisfy the condition to be a singular Bj\"{o}rling data as in the definition \ref{defn:singularBjorlingdata}.

Let  $t_{n_k}$ be as  in the definition \ref{lemma:f_n} and $\alpha_n^\prime= (\alpha_{n1}^\prime, \alpha_{n2}^\prime, \alpha_{n3}^\prime)$ and $\beta_n= (\beta_{n1}, \beta_{n2}, \beta_{n3})$.  At each $t_{n_k}\in I$, we have
\begin{align*}
     & \alpha_{n3}^\prime(t_{n_k})\lambda_1^\prime(t_{n_k})+\beta_{n3}(t_{n_k})\lambda_2^\prime(t_{n_k})\\
     &=\alpha_3^\prime(t_{n_k})\lambda_1^\prime(t_{n_k})+\beta_3(t_{n_k})\lambda_2^\prime(t_{n_k})+f_n(t_{n_k})(\alpha_3^\prime(t_{n_k})\lambda_1^\prime(t_{n_k})-\beta_3(t_{n_k})\lambda_2^\prime(t_{n_k}))
\end{align*}
Since for all $t\in I$  $\alpha_3^\prime\lambda_1^\prime+\beta_3\lambda_2^\prime=0$ and $f(t_{n_k})=0$ we have for all $t_{n_k}$,
\begin{equation}\label{lasteqn1}
    \alpha_{n3}^\prime(t_{n_k})\lambda_1^\prime(t_{n_k})+\beta_{n3}(t_{n_k})\lambda_2^\prime(t_{n_k})=0
\end{equation}

For each $t$, $\lambda_1^{\prime\prime}(t)\alpha_{3}^\prime(t)+\lambda_1^\prime(t)\alpha_{3}^{\prime\prime}+\lambda_2^\prime\beta_{3}^\prime(t)+\lambda_2^{\prime\prime}(t)\beta_{3}(t)=0$ and for each $t_{n_k}$, $f(t_{n_k})=0,$ we get
\[
\lambda_1^{\prime\prime}\alpha_{n3}^\prime+\lambda_1^\prime\alpha_{n3}^{\prime\prime}+\lambda_2^\prime\beta_{n3}^\prime+\lambda_2^{\prime\prime}\beta_{n3}(t_{n_k})= f_n^\prime(t_{n_k})\left(\beta_3(t_{n_k})\lambda_1^\prime(t_{n_k})-\alpha_3^\prime(t_{n_k})\lambda_2^\prime(t_{n_k})\right)\]

As $f^\prime(t_{n_k})\neq 0$ and using equation \ref{eqnmain2}, we have
\begin{equation}\label{eqnmain3}
  \forall t_{n_k},\;\; \lambda_1^{\prime\prime}\alpha_{n3}^\prime+\lambda_1^\prime\alpha_{n3}^{\prime\prime}+\lambda_2^\prime\beta_{n3}^\prime+\lambda_2^{\prime\prime}\beta_{n3}\neq 0
\end{equation}
 
 We have $\alpha_n^\prime(t_{n_k})=\alpha(t_{n_k})$ and $\beta_n(t_{n_k})= \beta(t_{n_k})$ therefore
\begin{equation}\label{maineqn4}
\text{ when } \alpha^\prime_n(t)\neq 0,\;\; D(\alpha_{n12}^\prime,\alpha_{n12}^{\prime\prime})\neq 0, \text{ when } \beta_n^\prime(t)\neq 0,\; D(\beta_{n12},\beta_{n12}^{\prime})\neq 0.
\end{equation}

Therefore at each $t_{n_k},\; k=0,1,.. ,a_n$,  (1), (2) and (3) conditions (equations \ref{lasteqn1}, \ref{eqnmain3}, \ref{maineqn4}) of proposition \ref{prop:Swallotail} satisfies.

As $\lambda$ is $1-1$ at these points, for each $n$  we get  $a_n+1$ swallowtails on each $X_n$.  Since $a_n>n$, we have increasing number of points  on  the trace of $\lambda$ such the maxface $X_n$ for the  data $\{\lambda, \alpha_n, \beta_n\}$  has increasing numbers of swallowtails.

Let fix some point $t_0\in I$ and we denote $X= (X^1, X^2, X^3)$, for $z\in \overline{\Omega}$, let $|\int_{ \lambda(t_0)}^zdw|<L$. Then $\forall z\in \overline{\Omega},$ we for each $j= 1,2,3$, we get
$$|X_n^j(z)-X^j(z)|= |Re\int_{\lambda(t_0)}^z \psi_n^j(w)-\phi^j(w) dw|$$
$$=\left|Re\int_{\lambda(t_0)}^z -ig_n(w)\right|$$
$$\leq \|g_n\|_\Omega.\left|\int_{\lambda(t_0)}^zdw\right|\leq \|g_n\|_\Omega L$$

Therefore we have
$$\|X_n-X\|_\Omega \leq \|g_n\|_\Omega L$$ and therefore $X_n\to X$ on $\overline{\Omega}$

Summarizing all above, we proved the following:

\begin{theorem}\label{thm:main}
Let $X$ be a maxface having generalized cone-like singularity on $trace$ of $\lambda$, with singular Bj\"{o}rling data $\{\lambda,\alpha, \beta\}$. Moreover the data  $\{\lambda,\alpha, \beta\}$ has a sequence of the scaling functions as in the definition \ref{lemma:f_n}. Then there is a sequence of maxfaces $X_n$ defined on a neighborhood $\Omega$ of trace of $\lambda$, having an increasing number of swallowtails, and sequence of maxfaces converges (in the norm $\|\;.\;\|_\Omega$) to $X$, having cone-like.
\end{theorem}

Many examples can be constructed using the theorem \ref{thm:main}. In particular, we have the following example similar to the one we have already seen \ref{examp:Seqnswallotail1}.

\begin{example}
Let $\lambda(t)= e^{it};\;\;0\leq t\leq 2\pi$, we start with the data $\alpha^\prime(t)$ and $\beta$ as in the example \ref{exampls:cone-likeonunitcircle}, so that every point on the unit circle is a cone-like singularity.  For this Bj\"{o}rling data, we have already seen there exists a sequence of scaling functions as in the example \ref{f_nexamplet}  To get a sequence of maxfaces having an increasing number of swallowtails; we start with the singular Bj\"{o}rling data  on $\lambda$ as

\begin{eqnarray*}
&\alpha_n^\prime(t)= \alpha^\prime(t)+f_n(t)\beta(t)=\left(\cos t+\frac{\cos nt \sin t}{n 2^{n+2}}\right)(\cos t, \sin t, -1)\\
&\beta_n(t)= \beta(t)-f_n(t)\alpha^\prime(t)=\left(\sin t-\frac{\cos nt \cos t}{n 2^{n+2}}\right)(\cos t, \sin t, -1)
\end{eqnarray*}

 We have $t_{n_k}= (2k+1)\frac{\pi}{2n},\;\; k=0,1,....,2n-1$, and $a_n= 2n-1$, so by the theorem \ref{thm:main}, every $X_n$ has at least $2n$ swallowtails on the unit circle and that converges to the cone-like.
\end{example}

\medskip
\bibliography{Refs}

\begin{bibdiv}
\begin{biblist}

\bib{RPR2016}{article}{
      author={Dey, R.},
      author={Kumar, P.},
      author={Singh, R.~K.},
       title={Existence of maximal surface containing given curve and special
  singularity},
        date={2018},
     journal={Journal of the Ramanujan Mathematical Society},
      volume={33},
      number={4},
       pages={455\ndash 471},
}

\bib{Estudillo1992}{article}{
      author={Estudillo, F. J.~M.},
      author={Romero, A.},
       title={Generalized maximal surfaces in lorentz–minkowski space l3},
        date={1992},
     journal={Mathematical Proceedings of the Cambridge Philosophical Society},
      volume={111},
      number={3},
       pages={515–524},
}

\bib{Fernandez2005a}{article}{
      author={Fernández, I.},
      author={López, F.~J.},
      author={Souam, R.},
       title={The space of complete embedded maximal surfaces with isolated
  singularities in the 3-dimensional lorentz-minkowski space},
        date={2005},
        ISSN={1432-1807},
     journal={Mathematische Annalen},
      volume={332},
      number={3},
       pages={605\ndash 643},
         url={https://doi.org/10.1007/s00208-005-0642-6},
}

\bib{fujimori2015}{article}{
      author={Fujimori, S.},
      author={Kim, Y.W.},
      author={Koh, S.-E.},
      author={Rossman, W.},
      author={Shin, H.},
      author={Umehara, M.},
      author={Yamada, K.},
      author={Yang, S.-D.},
       title={Zero mean curvature surfaces in lorentz--minkowski 3-space which
  change type across a light-like line},
        date={201501},
     journal={Osaka J. Math.},
      volume={52},
      number={1},
       pages={285\ndash 299},
         url={https://projecteuclid.org:443/euclid.ojm/1427202882},
}

\bib{Fujimori2009}{article}{
      author={Fujimori, S.},
      author={Rossman, W.},
      author={Umehara, M.},
      author={Yamada, K.},
      author={Yang, S.-D.},
       title={New maximal surfaces in minkowski 3-space with arbitrary genus
  and their cousins in de sitter 3-space},
        date={2009},
        ISSN={1420-9012},
     journal={Results in Mathematics},
      volume={56},
      number={1},
       pages={41},
         url={https://doi.org/10.1007/s00025-009-0443-4},
}

\bib{Fujimori2007}{article}{
      author={Fujimori, S.},
      author={Saji, K.},
      author={Umehara, M.},
      author={Yamada, K.},
       title={Singularities of maximal surfaces},
        date={2007},
        ISSN={1432-1823},
     journal={Mathematische Zeitschrift},
      volume={259},
      number={4},
       pages={827},
         url={https://doi.org/10.1007/s00209-007-0250-0},
}

\bib{Kim2006}{article}{
      author={Kim, W.},
      author={Yang, S.-D.},
       title={A family of maximal surfaces in lorentz-minkowski three-space},
        date={2006},
     journal={Proc. Amer. Math. Soc.},
      number={134},
       pages={3379\ndash 3390},
}

\bib{Kim2007}{article}{
      author={Kim, Y.~W.},
      author={Yang, S.-D.},
       title={Prescribing singularities of maximal surfaces via a singular
  björling representation formula},
        date={2007},
        ISSN={0393-0440},
     journal={Journal of Geometry and Physics},
      volume={57},
      number={11},
       pages={2167 \ndash  2177},
}

\bib{KOBAYASHI1983}{article}{
      author={Kobayashi, O.},
       title={Maximal surfaces in the 3-dimensional minkowski space $l^3$},
        date={198312},
     journal={Tokyo J. Math.},
      volume={06},
      number={2},
       pages={297\ndash 309},
         url={https://doi.org/10.3836/tjm/1270213872},
}

\bib{KOBAYASHI1984}{article}{
      author={Kobayashi, O.},
       title={Maximal surfaces with conelike singularities},
        date={198410},
     journal={J. Math. Soc. Japan},
      volume={36},
      number={4},
       pages={609\ndash 617},
         url={https://doi.org/10.2969/jmsj/03640609},
}

\bib{teramoto2020gaussian}{article}{
      author={Teramoto, K.},
       title={On gaussian curvatures and singularities of gauss maps of
  cuspidal edges},
        date={2020},
      eprint={2003.10645},
}

\bib{UMEHARA2006}{article}{
      author={Umehara, M.},
      author={Yamada, K.},
       title={Maximal surfaces with singularities in minkowski space},
        date={200602},
     journal={Hokkaido Math. J.},
      volume={35},
      number={1},
       pages={13\ndash 40},
         url={https://doi.org/10.14492/hokmj/1285766302},
}

\end{biblist}
\end{bibdiv}

\end{document}